\documentclass{enarticle}

\newcommand{\Cinf}{C^{\infty}}
\newcommand{\dpi}{\: \mathrm{d}\pi}
\newcommand{\grad}{\nabla}

\let\div\relax \DeclareMathOperator{\div}{div}

\DeclareMathOperator{\Ran}{Ran}
\newcommand{\rapi}{\ra_{\pi}}
\newcommand{\HC}{\cH_{\C}}
\newcommand{\Lpi}{\cL^2(\pi)}
\newcommand{\Hpi}{\cH^2(\pi)}
\newcommand{\Wpi}{\cW^2(\pi)}
\newcommand{\pidx}{\: \pi(\ddx)}
\newcommand{\normpi}[1]{\|#1\|_{\pi}}

\title{Variance reduction for diffusions}
\author[1]{Chii-Ruey \textsc{Hwang}}
\author[1]{Raoul \textsc{Normand}}
\author[2]{Sheng-Jhih \textsc{Wu}}
\location[1]{Institute of Mathematics, Academia Sinica, Taipei 10617, Taiwan}
\location[2]{Center for Advanced Statistics and Econometrics Research, School of Mathematical Sciences, Soochow University, Suzhou 215006, China}
\email{crhwang@sinica.edu.tw; rnormand@math.sinica.edu.tw; szwu@suda.edu.cn}

\begin{document}

\maketitle

\begin{abstract}
The most common way to sample from a probability distribution is to use Markov Chain Monte Carlo methods. One can find many diffusions with the target distribution as equilibrium measure, so that the state of the diffusion after a long time provides a good sample from that distribution. One naturally wants to choose the best algorithm. One way to do this is to consider a reversible diffusion, and add to it an antisymmetric drift which preserves the invariant measure. We prove that, in general, adding an antisymmetric drift reduces the asymptotic variance, and provide some extensions of this result.
\end{abstract}


\MSC Primary: 60J60, 65B99; Secondary: 60J35, 47D07

\section{Introduction}

Sampling directly from a probability distribution is often infeasible in practice, in particular when the distribution is only known up to a multiplicative constant. Markov Chain Monte Carlo (MCMC) methods are popular for obtaining samples from a target probability $\pi$, see e.g. \cite{BrooksMCMC,HastingsMCMC}. The idea is to construct a Markov chain or Markov process $(X_t)$, whose invariant distribution is $\pi$. Under reasonable assumptions (``ergodicity''), the distribution of $(X_t)$ converges to $\pi$, so that, for $t$ large enough, $X_t$ provides a good sample of $\pi$. It is natural that one would prefer the dynamics that attain the desired equilibrium distribution as fast as possible. There are several ways to measure how good an algorithm is: mixing time, spectral gap, asymptotic variance... Hence, comparing the performance of different algorithms also depends on the criterion used.

The discrete case has been extensively studied in the literature, see e.g. \cite{ChenOTM, FrigessiOSS, MiraEFSS, MiraOIP, MiraNRMC, PeskunOMCS}. In general, one can choose reversible or non-reversible algorithms. The latter usually enjoy a faster convergence to equilibrium, but are more challenging to study. More specifically, of special interest to us is the scheme of modifying a given reversible process by a suitable antisymmetric perturbation so that the equilibrium is preserved, but the convergence is accelerated. All evidence points to this strategy always providing a better algorithm, be it in the discrete case \cite{ChenARMC} or in the continuous one \cite{HwangAGD,HwangAD,HwangABM}. In the continuous setting, the performance is understood in terms of the spectral gap of the generator of the diffusion. Our goal in this paper is to show that perturbing a diffusion by an antisymmetric drift also reduces the asymptotic variance. Let us note that the recent paper \cite{Rey-BelletILS} provides a similar result in terms of large deviations, with a corollary for the asymptotic variance. However, the setting is limited to the $d$-dimensional torus, and some details seem to be lacking in the proof.

Let us informally introduce our setting. Let $U$ be a given energy function and $\pi$ the probability with density proportional to $e^{-U(x)}$. If one is interested in sampling from $\pi$, then a usually used diffusion is the time reversible Langevin equation with $\pi$ as equilibrium measure
\begin{equation} \label{eq:diff}
\rmd X_t = \sqrt 2 \: \rmd W_t - \grad U(X_t) \dt,
\end{equation}
where $(W_t)$ is a Brownian motion. Perturbing the reversible diffusion \eqref{eq:diff} by adding an 
antisymmetric drift term results in
\begin{equation} \label{eq:diff2}
\rmd X_t = \sqrt 2 \: \rmd W_t - \grad U(X_t) \dt + C(X_t) \dt,
\end{equation}
where the vector field $C$ is weighted divergence-free with respect to $\pi$, i.e., $\div(Ce^{-U}) = 0$. This ensures that the non-reversible diffusion \eqref{eq:diff2} also has equilibrium $\pi$. That there are many ways to choose such a perturbation $C$, such as taking $C = Q \grad U$ for an antisymmetric matrix $Q$. Note that, in any case, it is unnecessary to know the normalization constant for $\pi$.

Let $-L$ (we use the sign convention to make $L$ positive) denote the infinitesimal generator of \eqref{eq:diff}. More precisely,
\[
L = - \Delta + \grad U \cdot \grad
\]
on, say, $\Cinf_c$, the space of infinitely differentiable functions with compact support. Informally, one may recover that $\pi$ is invariant by noticing that the adjoint operator is $L^* f = - \Delta f - \div(f \grad U)$, so that $L^* \pi = 0$.  This process is reversible, which amounts to saying that $L$ is symmetric in $\Lpi$, the space of square-integrable complex functions with respect to $\pi$. On the other hand, the generator of the modified equation is given by $- L_C$, where
\[
L_C = L - C \cdot \grad.
\]
It has the adjoint $L_C^* f = L^* f + \div (f C)$. Hence, to ensure that the diffusion \eqref{eq:diff2} also has $\pi$ as its invariant measure, 
it is necessary to assume that $L_C^* \pi = 0$, i.e., $ \div(Ce^{-U}) = 0$. In \cite{HwangAD} (see also \cite{HwangAGD}), the spectral gap of these operators is employed as a measurement of the rate of convergence of these algorithms. It is shown therein that $L_C$ has a larger spectral gap than $L$, which is just a way to say that the irreversible algorithm performs better than the reversible one.

Essentially, the spectral gap measures the exponential rate in the convergence of the distribution of $(X_t)$ to $\pi$. We are interested here in comparing these algorithms in terms of asymptotic variance, which rather measures the speed of convergence of the \emph{average} of $f(X_t)$ to the mean $\int f \dpi$. Without dwelling on the technical details, which will be tackled in Section \ref{sec:CLT}, let us explain our results, and the motivation for the next section. Assume that $(X_t)_{t \geq 0}$ is an ergodic Markov process with equilibrium measure $\pi$. Let $\Lpi$ be the space of functions which are square-integrable with respect to $\pi$, with inner product $\la \cdot, \cdot \rapi$. Denote by $-G$ the generator of $(X_t)$ in $\Lpi$, with domain $\cD$. Take $f \in \Lpi$ and assume that there is a solution $h \in \cD$ to the Poisson equation
\begin{equation} \label{eq:Poissoneq}
G h = f.
\end{equation}
Then a central limit theorem holds: 
\[
\sqrt{t} \left ( \frac1t \int_0^t f(X_s) \ds - \int f \dpi \right ) \convlaw \cN(0, \s^2(f))
\]
where
\begin{equation} \label{eq:asvar}
\s^2(f) = 2 \la f, h \rapi
\end{equation}
is called the asymptotic variance (associated with $f$). This explains why the asymptotic variance is a natural gauge of the efficiency of the MCMC algorithm. We refer the readers to \cite{BhattaCLT} or Chapter 2 of \cite{Komorowski} for a more detailed discussion.

The main goal of the present study is to prove that, in the sense of asymptotic variance, the irreversible diffusion \eqref{eq:diff2} converges to equilibrium faster than \eqref{eq:diff}. More precisely, if $\s^2_C(f)$ and $\s_0^2(f)$ denote the corresponding asymptotic variances, we will show that, under mild conditions,
\begin{equation} \label{eq:ineqvar}
\s^2_C(f) \leq \s_0^2(f).
\end{equation}
The precise results are Theorem \ref{th:onRd} and \ref{th:onmanifold}. From \eqref{eq:Poissoneq} and \eqref{eq:asvar}, this amounts to proving that
\begin{equation} \label{eq:ineqop0}
\la L_C^{-1} f, f \rapi \leq \la L^{-1} f, f \rapi,
\end{equation}
and that is merely a result on operators. For this reason, the paper will be organized as follows. In Section \ref{sec:general}, we will provide conditions on some operators for \eqref{eq:ineqop0} to hold. In Section \ref{sec:ex}, we show that, under mild conditions, the diffusions \eqref{eq:diff} and \eqref{eq:diff2} are well-behaved, that the CLT holds, and that their generators enjoy properties that ensure \eqref{eq:ineqop0}. In short, we prove \eqref{eq:ineqvar}, i.e. that the irreversible diffusion \eqref{eq:diff2} converges to $\pi$ faster than the reversible one \eqref{eq:diff}. We study first the case of diffusions on $\R^d$, and then on a compact Riemannian manifold. Section \ref{sec:extensions} provides some extension of our main results: we characterize the cases of equality in \eqref{eq:ineqvar}, study the worst-case analysis, and finally the behavior of $\s^2_C(f)$ when the amplitude of the drift grows, as is done in \cite{FrankeBSG,HwangABM}. Concluding remarks are offered in Section \ref{sec:concl}. 

\section{General results} \label{sec:general}

\subsection{Spectral theorem}

We will rely extensively on the spectral theorem, so let us recall some of its features that we will use. For the results and definitions, see Chapter VIII in \cite{RS1}. Consider a \emph{complex} Hilbert space $\cH$ with an inner product $\la \cdot, \cdot \ra$, and a self-adjoint operator $T$ with domain $\cD$. Let $\s(T) \subset \R$ be its spectrum. Then the following hold.
\begin{enumerate}
\item For any $z \in \C \bsl \R$, we can define the resolvent $(I-z T)^{-1}$, which is a bounded operator on $\cH$.
\item More generally, for any bounded Borel function $\phi \: : \: \s(T) \to \C$, one can define the bounded operator $\phi(T)$, and $\phi \mapsto \phi(T)$ is an algebra homomorphism.
\item For any $v \in \cH$, there exists a Borel measure $\mu_v$, called the spectral measure associated with $v$, such that for any bounded Borel function $\phi \: : \: \s(T) \to \C$,
\[
\la v, \phi(T)(v) \ra = \int_{\s(T)} \phi(s) \: \mu_v(\dds).
\]
In particular $\mu_v(\s(T)) = \la v, v \ra$.
\end{enumerate}

\subsection{Statement of the main inequality}

Recall from the introduction that, informally, we want to prove \eqref{eq:ineqop0}. We shall establish this inequality under the general assumptions \textbf{(G1)}, \textbf{(G2)} and \textbf{(G3)} below. It is worth noticing that this result says that, formally, adding an antisymmetric perturbation to an reversible Markov process decreases the asymptotic variance. It however relies on fine properties of the generators, which are usually hard to check. Hence, we feel that a general result concerning Markov processes would be a mere technicality with no actual application. Instead, we shall devote Section \ref{sec:ex} to proving that, under mild assumptions on $U$ and $C$, the generators of the diffusions \eqref{eq:diff} and \eqref{eq:diff2} do enjoy properties \textbf{(G1)}, \textbf{(G2)} and \textbf{(G3)}, so that the result applies.

We obviously want to consider real functions. Hence, we fix a \emph{real} Hilbert space $\cH$ with an inner product $\la \cdot, \cdot \ra$. Take $S$ to be an (unbounded) operator on $\cH$ with domain $\cD$, $A$ be another operator whose domain contains $\cD$, and let $S_A = S + A$ (with domain $\cD$). 
We assume the following throughout this section.
\begin{itemize}
\item[\textbf{(G1)}] $S$ is symmetric and positive, and $A$ is antisymmetric;
\item[\textbf{(G2)}] $S$ is a bijection from $\cD$ onto $\cH$ with a bounded inverse;
\item[\textbf{(G3)}] $S_A$ and $S_{-A}$ are bijections from $\cD$ onto $\cH$.
\end{itemize}
We shall prove the following.
\begin{theorem} \label{th:ineqop}
Assume \textbf{(G1)}, \textbf{(G2)} and \textbf{(G3)}. Then, for any $f \in \cH$,
\[
\la S_A^{-1} f, f \ra \leq \la S^{-1} f, f \ra.
\]
\end{theorem}

The proof of this result is essentially the simple computation shown in Section \ref{sec:maincomp}. The goal of our assumptions, however, is to obtain nice properties for some operators, as presented in the next section, which allow to rigorously justify that computation.

As mentioned, we want to consider real functions. However, the spectrum of our operators will lie in the complex plane, and we will thus need to use complex functions. We then consider $\HC$ the complexification of $\cH$, with inner product still written $\la \cdot, \cdot \ra$, see e.g. \cite{FrankeBSG}. Unlike in this reference though, we take the inner product to be sesquilinear on the left, so as to have the same convention as in \cite{RS1}. Clearly, we may write $\cH \subset \HC$, and all the operators $S$, $S_A$, $S^{-1}$, $S_A^{-1}$ and $V$ (defined below) send $\cH$ to $\cH$.

\subsection{Preparation}

According to the assumptions, $S^{-1}$ is a well-defined bounded operator from $\HC$ onto $\cD$, which is symmetric and positive, as $S$ is. Hence, it has a unique symmetric bounded square root $V = S^{-1/2}$. We define $\cR = V(\HC)$ to be the range of $V$. 
Notice that $V(\cR) = S^{-1}(\HC) = \cD$, so that, to summarize, we have the following operators:
\begin{itemize}
\item $S$ is a symmetric bijection from $\cD$ onto $\HC$;
\item $S^{-1}$ is a bounded self-adjoint bijection from $\HC$ onto $\cD$;
\item $V$ is a bounded self-adjoint bijection from $\HC$ onto $\cR$;
\item the restriction of $V$ to $\cR$ is a bijection onto $\cD$.
\end{itemize}
In particular, $\cD \subset \cR$, so that $\cR$ is dense in $\HC$. An easy consequence of these properties is also that, when $f \in \cR$, then $g = V S V f$ is well-defined, and
\[
V g = V V S V f = S^{-1} S V f = V f
\]
so that, since $V$ is a bijection,
\begin{equation} \label{eq:VSV}
V S V f = f, \quad f \in \cR.
\end{equation}
This is essentially saying that $V$ commutes with any operator that commutes with $S^{-1}$, so with $S$. 

Now, it makes sense to define
\[
B = i V A V
\]
with domain $\cR$. As we mentioned, $\cR$ is dense in $\HC$, so $B$ is a densely-defined unbounded operator on $\HC$. Moreover, since $V$ is symmetric and $A$ antisymmetric,
\[
B^* \supset - i (V)^* A^* (V)^* \supset - i V (- A) V = B,
\]
so that $B$ is symmetric. The main reason for \textbf{(G3)} is that it allows to prove the much stronger following result.

\begin{lemma} \label{lem:esa}
The operator $B$ is essentially self-adjoint.
\end{lemma}

\begin{proof}
We just saw that $B$ is symmetric, so all we need to check is that the range $\Ran(B \pm i)$ of $B \pm i$ is dense, see the corollary to Theorem VIII.3 in \cite{RS1}. Let us first prove that $\Ran (B + i) \supset \cR$. Fix $g \in \cR$. Computing $S V (-i g)$ provides a well-defined element of $\HC$, and since $S_A$ is onto $\HC$, there exists a $h \in \cD$ such that $S_A h = S V (-i g)$. We are finally allowed to define $f = V S h$. Now we have $(B + i)f = i(V A V f + f)$, and
\begin{align*}
V A V f & = V (S_A - S) V f \\
& = V S_A V f - V S V f \\
& = V S_A V V J h - f \\
& = V S_A h - f \\
& = V S V (- ig) - f \\
& = -i g - f,
\end{align*}
where we use \eqref{eq:VSV} twice. We conclude that $(B + i) f = i(- ig - f + f) = g$.

Similarly, we can check that $\Ran (B - i) \supset \cR$, by considering $h \in \cD$ such that $S_{-A} h = S V (i g)$ instead. This is possible since $S_{-A}$ is onto $\HC$. The same computation shows that
\[
V A V f = f - ig,
\]
provided we write this time $A = S - S_{-A}$, so finally $(B-i)f = g$, and the result follows.
\end{proof}

Now, we want to give a formula for $S_A^{-1} = (S+A)^{-1}$. Informally, we may write
\[
(S+A)^{-1} = \left ( S^{1/2}(I + S^{-1/2}AS^{-1/2})S^{1/2} \right )^{-1} = S^{-1/2}(I + VAV)^{-1} S^{-1/2} = V(I -iB)^{-1}V.
\]
There are two issues in this computation: first $S^{1/2}$ is not a very pleasant object; and secondly $(I -iB)^{-1}$ is not defined. However, this second point can be dealt with as follows. Consider the closure of $B$, which we still denote $B$. It has domain $\bar{\cR}$, the closure of $\cR$ for the norm $\| \cdot \| + \| B \cdot \|$, and is self-adjoint by Lemma \ref{lem:esa}. Hence, by the spectral theorem, the resolvent $(I - i B)^{-1}$ is well-defined, and it is a bounded operator on $\HC$. For the first point, note that we actually do not need to talk about $S^{1/2}$: since $S_A$ is invertible by assumption, we just have to check that the suggested formula provides a (left- or right-) inverse for $S_A$, which we do below.

\begin{lemma} \label{lem:inverse}
We have the equality
\[
S_A^{-1} = V (I - i B)^{-1} V.
\]
\end{lemma}

\begin{proof}
As mentioned, this is just simple algebra. Compute, for $f \in \cD$,
\begin{align*}
V (I - i B)^{-1} V S_A f & = V (I - i B)^{-1} V (S + A) V V S f \\
& = V (I - i B)^{-1} (V S V + V A V) V S f \\
& = V (I - i B)^{-1} (I - i B) V S f \\
& = V V S f \\
& = f,
\end{align*}
where all the calculations are justified by the domain assumptions and \eqref{eq:VSV}. This shows that $V (I - i B)^{-1} V$ is the left-inverse, so the inverse of $S_A$.
\end{proof}

\subsection{Main computation} \label{sec:maincomp}

Let us now end the proof of Theorem \ref{th:ineqop}. The essential point is the following computation. Informally, it only relies on \textbf{(G1)}. Assumptions \textbf{(G2)} and \textbf{(G3)} allow to obtain the results of the earlier sections, which rigorously justify this computation.

Let $f \in \cH$ (not $\HC$) and $g = Vf$. Let $\mu_g$ be the spectral measure of $B$ associated with $g$. Using Lemma \ref{lem:inverse} and the spectral theorem, we may thus compute
\begin{align*}
\la S_A^{-1} f, f \ra
& = \la V (I - i B)^{-1} V f, f \ra \\
& = \la (I - i B)^{-1} g, g \ra \\
& = \int_{\s(B)} \frac{1}{1 - i y} \mu_g(\ddy) \\
& = \int_{\s(B)} \frac{1 + iy}{1 + y^2} \mu_g(\ddy) \\
& = \int_{\s(B)} \frac{1}{1 + y^2} \mu_g(\ddy),
\end{align*}
where the last inequality is justified by the fact that we consider real quantities, or alternatively that $\mu_g$ is symmetric. Finally,
\[
\int_{\s(B)} \frac{1}{1 + y^2} \mu_g(\ddy) \leq \int_{\s(B)} 1 \: \mu_g(\ddy) = \mu_g(\s(B)) = \la g, g \ra = \la V f, V f \ra = \la S^{-1} f, f \ra.
\]
Theorem \ref{th:ineqop} follows.

\section{Two examples of diffusion processes} \label{sec:ex}

\subsection{Assumptions}

We shall now apply the results obtained in the previous section
to the generators of the diffusions \eqref{eq:diff} and \eqref{eq:diff2} on a space $M$. We will consider two specific cases: either $M = \R^d$, or $M$ is a smooth compact connected $d$-dimensional Riemannian manifold. We will make the natural following assumptions.
\begin{description}
\item[\textbf{(A1)}] $U \, : \, M \to \R$ is $C^2$ and $C \: : \: M \to M$ is $C^1$;
\item[\textbf{(A2)}] $\int_M e^{-U(x)} \dx < \infty$;
\item[\textbf{(A3)}] $\div(C e^{-U}) = 0$.
\end{description}
To simplify, in the manifold case, we will even assume
\begin{description}
\item[\textbf{(A1')}]  $U \, : \, M \to \R$ and $C \: : \: M \to M$ are smooth.
\end{description}
The first two conditions \textbf{(A1)} / \textbf{(A1')} and \textbf{(A2)} are just natural regularity and integrability assumptions. As we mentioned, \textbf{(A3)} is needed for $\pi$ to be invariant in \eqref{eq:diff2}. In the compact manifold case, there is no issue with explosion of the diffusion, or boundedness of the functions considered, and that is all we shall assume. On $\R^d$, we will need certain growth conditions on $U$ and $C$.

Let us introduce some common notations. We denote, as in the introduction, $\Lpi$ as the space of complex functions on $M$ which are square-integrable with respect to $\pi$. It is a Hilbert space endowed with the inner product
\[
\la f, g \rapi = \int_M \bar{f} g \dpi,
\]
and norm denoted by $\normpi{\cdot}$. For $k \in \N \cup \{{\infty}\}$, let $C^k_c$ be the space of $k$-times differentiable functions with compact support. For $m \geq 0$, we define the weighted Sobolev space $\cH^m(\pi)$ as the completion of $\Cinf_c$ with respect to the norm induced by the inner product
\[
\la f , g \ra_{\cH^m(\pi)} = \sum_{|a| \leq m} \la \partial_{a} f, \partial_{a} g \rapi.
\]
Alternatively, on $\R^d$, it can be defined as the subspace of functions whose $m$ first derivatives (in the distributional sense) are in $\Lpi$, by Meyers-Serrin theorem. 
Let $\un$ be the constant function equal to $1$. For any subspace $\cX$ of $\Lpi$, we set
\[
\cX_0 = \left \{ f \in \cX, \, \int f \dpi  = 0 \right \} = \cX \cap \un^{\perp}
\]
the subspace of $\cX$ of functions with zero mean. Finally, we say that an unbounded operator $T$ has a spectral gap if the real part of all the values in its spectrum are bounded away from 0 (so in particular $T^{-1}$ is well-defined and bounded).

Finally, let us define the operators
\begin{equation} \label{eq:L}
L = - \Delta + \grad U \cdot \grad
\end{equation}
and
\begin{equation} \label{eq:LC}
L_C = L - C \cdot \grad.
\end{equation}
These formulas a priori only define distributions. If the domain is well-chosen, we will see that we are actually able to consider them as operators on $\Lpi$. In any case, note that, by integration by parts, we always have that, for $f,g \in \Cinf_c$,
\begin{equation} \label{eq:Lfg}
\la L f, g \rapi = \la \grad f, \grad g \rapi = \la f, L g \rapi
\end{equation}
and
\begin{equation} \label{eq:Cdotgrad}
\la C \cdot \grad f, g \rapi = - \la f, C \cdot \grad g \rapi,
\end{equation}
where the last computation uses \textbf{(A3)}.

\subsection{Central limit theorem} \label{sec:CLT}

As should be clear from Section \ref{sec:general}, so as to apply Theorem \ref{th:ineqop}, we need precise properties of operators, the operators in question being the generators of \eqref{eq:diff} and \eqref{eq:diff2}. However, what do we mean by generator here? In general, a Markov process $(X_t)$ is associated with a $C_0$-semigroup $(P_t)$ via $P_t f(x) = \E_x(f(X_t))$, though the precise definitions of these terms and the corresponding results are far from canonical. In any case, the Hilbert structure that we wish to use, in view of Section \ref{sec:general}, is not clear here. Following Chapter 2 in \cite{Komorowski}, let us shortly explain how this is done. We shall employ here the definitions of \cite{RW1} and \cite{RW2}.

Take $E$ a complete separable metric space, and $(P_t)$ a conservative Feller-Dynkin semigroup on $C_0$, the space of continuous functions vanishing at infinity. From this, one can define a strong Markov process with transition function $(P_t)$, see Chapter III in \cite{RW1}. Let us also assume that $\pi$ is an invariant probability, i.e. that
\[
\int P_t f (x) \pidx = \int f(x) \pidx
\]
for all $f \in C_0$ and $t \geq 0$. Then
\[
\normpi{P_t f}^2 = \int (P_t f(x))^2 \pidx \leq \int P_t f^2(x) \pidx = \int f^2(x) \pidx = \normpi{f}^2,
\]
so that it is a contraction for the $\normpi{\cdot}$ norm. Since $C_0$ is dense in $\Lpi$ (say, because it is dense in the usual $\cL^2(\ddx)$), then $P_t$ can be extended to a contraction on the whole of $\Lpi$. Then, as a semigroup on $\Lpi$, $(P_t)$ is strongly continuous: for $\eps > 0$ and $f \in \Lpi$, take $g$ in $C_0$ such that $\normpi{f-g} \leq \eps$, and $t > 0$ such that $\|P_s g - g\|_{\infty} \leq \eps$ for $s \leq t$, what is possible since $(P_t)$ is strongly continuous on $C_0$. Then, for $s \leq t$,
\begin{align*}
\normpi{P_s f - f} & \leq \normpi{P_s f - P_s g} + \normpi{P_s g - g} + \normpi{g - f} \\
& \leq \normpi{f-g} + \|P_s g - g\|_{\infty} + \eps \\
& \leq 3 \eps,
\end{align*}
so that $(P_t)$ is strongly continuous on $\Lpi$. Hence, we can define its generator $G$, with domain $\cD(G)$, the set of $f \in \Lpi$ such that $(P_t f - f)/t \to g$ in $\Lpi$ for some $g \in \Lpi$. This is what we mean when we talk about generators on $\Lpi$. 

Now, assume additionally that $\pi$ is ergodic, in that the shift invariant $\s$-field is trivial. Then, there is a CLT: for any real $f \in \Lpi_0$ such that there exists a solution $h \in \cD(G)$ to the Poisson equation $-G h = f$,
\[
t^{-1/2} \int_0^t f(X_s) \ds
\]
converges under $\P_{\pi}$ to a centered Gaussian variable with variance $\s^2(f) = 2 \la f, h \rapi$. Note that we assume that $f \in \Lpi_0$, i.e. that it has mean 0, which clearly does not make any difference, but circumvents too heavy notation.

With a couple differences, this result can be found in the following references.
\begin{itemize}
\item At the beginning of Chapter 2 in \cite{Komorowski}. The authors assume that $(P_t)$ is strongly continuous on $\Lpi$, which, as we saw, is unnecessary; moreover, one readily checks from the proof that the additional domain assumptions are not needed here.
\item In Theorem 2.1 of \cite{BhattaCLT}. The spaces called $\B$ and $\B_0$ therein are the same since we already assume that $(P_t)$ is strongly continuous on $C_0$. Moreover, the result is given for $t = n \in \N$, but readily extends for continuous $t$, since $n^{-1/2} \int_n^{n+1} f(X_s) \to 0$ in $\Lpi$.
\end{itemize}
 
\subsection{On $\R^d$}

\subsubsection{Assumptions and results}

Let us first consider the case of a diffusion on $\R^d$. As mentioned, we shall make the following extra growth assumptions on $U$ and $C$:
\begin{itemize}
\item[\textbf{(A4)}] for all $\eps > 0$, there is a $c_{\eps} > 0$ such that
\[
|C \cdot \grad U| + |D^2 U| \leq \eps |\grad U|^2 + c_{\eps},
\]
where $D^2 U$ is the Hessian matrix of $U$ and $|\cdot|$ is any norm;
\item[\textbf{(A5)}] there is a constant $K$ such that
\[
|C| \leq K ( |\grad U| + 1);
\]
\item[\textbf{(A6)}] as $x \to \infty$,
\[
|\grad U(x)| \to \infty.
\]
\end{itemize}
The assumption \textbf{(A5)} tells us that the perturbation is not too large compared to the drift. In particular, it is verified in the case when $C = Q \grad U$ where $Q$ is a skew-symmetric matrix. These assumptions ensure that the generators have nice properties, which will allow to apply the results from Theorem \ref{th:ineqop}. The first result concerns the well-posedness of \eqref{eq:diff2}, and properties of its generator, which apply in particular to \eqref{eq:diff}. This is merely a compilation from results in the literature.

\begin{theorem} \label{th:onRd0}
Assume \textbf{(A1)}, \textbf{(A2)}, \textbf{(A3)} , \textbf{(A4)}, \textbf{(A5)}. Then the following hold.
\begin{enumerate}
\item Equation \eqref{eq:diff2} has a unique strong solution, and this solution is not explosive.
\item The measure $\pi$ is its unique invariant distribution.
\item The generator of \eqref{eq:diff2} on $\Lpi$ is $- L_C$, with domain $\Hpi$.
\end{enumerate}
\end{theorem}

Our main result deals with the fact that \emph{adding an antisymmetric drift reduces the asymptotic variance}, as made precise in the following result. Once again, we assume without loss of generality that $f$ has mean 0.

\begin{theorem} \label{th:onRd}
Assume \textbf{(A1)}, \textbf{(A2)}, \textbf{(A3)}, \textbf{(A4)}, \textbf{(A5)}, \textbf{(A6)}. Then $L_C$ is onto $\Lpi_0$, and for $f \in \Lpi_0$ and $h \in \Hpi$ such that $L_C h = f$, the CLT holds:
\[
t^{-1/2} \int_0^t f(X_s) \ds
\]
converges to a centered normal variable with variance $\sigma^2_C(f)$, where
\[
\sigma^2_C(f) = 2 \la f, h \rapi.
\]
Additionally, adding an antisymmetric drift reduces the asymptotic variance, in that
\[
\sigma^2_C(f) \leq \sigma_0^2(f)
\]
for all $f \in \cL^2(\pi)_0$.
\end{theorem}

\subsubsection{Proof of Theorem \ref{th:onRd0}}

To begin with, let us check that $L_C$, as defined by Equation \eqref{eq:LC} on $\Hpi$, does define an element of $\Lpi$. For $f \in \Hpi$, we have $\Delta f \in \Lpi$, but we would also like $(\grad U + C) \cdot \grad f \in \Lpi$, and it is not clear what this last condition means, or how it depends on $C$. It turns out that this condition is redundant, provided that \textbf{(A5)} holds, along with a weak form of \textbf{(A4)}, namely that
\begin{equation} \label{eq:condL2}
|\Delta U| \leq a |\grad U|^2 + b
\end{equation}
for some $a < 1$ and $b \geq 0$. This is a modification of Lemma 2.1 in \cite{Lunardi}, which can be found in Lemma 7.1 in \cite{MetafuneRegularity}. We give a proof since this is a central point and that it is short in this $\Lpi$ setting.

\begin{lemma} \label{lem:cont}
Assume \eqref{eq:condL2} for some $a < 1$ and $b \geq 0$. Then the map $\phi \mapsto |\grad U| \phi$ is bounded from $\cH^1(\pi)$ to $\Lpi$. In particular, $L$ is bounded from $\Hpi$ to $\Lpi$, and so are $C \cdot \grad$ and $L_C$ if \textbf{(A5)} holds.
\end{lemma}

\begin{proof}
Since $\Cinf_c$ is dense is $\Hpi$, it suffices to prove the result for $\phi \in \Cinf_c$. Fix $\eps > 0$ such that $a + \eps  < 1$. By integration by parts, Cauchy-Schwartz inequality, and $|xy| \leq (x^2/\eps + \eps y^2)/2$, we have
\begin{align*}
\normpi{|\grad U| \phi}^2 & = - \int \phi^2 \grad U \cdot \grad e^{- U} \dx \\
& = \int \div \left ( \phi^2 \grad U \right ) e^{-U} \dx \\
& = \int 2 \phi \grad \phi \cdot \grad U \dpi + \int \Delta U \phi^2 \dpi \\
& \leq 2 \int \phi |\grad \phi| |\grad U| \dpi + \int (a |\grad U|^2 + b) \phi^2 \dpi \\
& \leq \frac{1}{\eps} \normpi{\grad \phi}^2 + (a + \eps) \normpi{\phi |\grad U|}^2 + b  \normpi{\phi}^2,
\end{align*}
whence the first part of the result follows after rearranging. The statement for $L$ follows readily, and the result for $L_C$ follows by assumption \textbf{(A5)}.
\end{proof}

Speaking about $L$ and $L_C$ as operators on $\Lpi$ with domain $\Hpi$ thus makes sense. To obtain more properties of these operators, let us first mention that our assumptions \textbf{(A1)}, \dots, \textbf{(A5)} are exactly those of Section 7 in \cite{MetafuneRegularity}. Theorem 7.4 and Remark 7.6 therein tell that
\begin{enumerate}
\item $- L_C$ is the generator of a $C_0$-semigroup $(T_t)$;
\item $(T_t)$ is strongly Feller;
\item $\pi$ is its unique invariant measure.
\end{enumerate}
In particular, since $\un \in \Hpi$, we have $\dfdt T_t \un = T_t L_C \un = 0$, so that $T_t \un = T_0 \un = \un$ in $\Lpi$. In particular, as a $C_0$-semigroup, $(T_t)$ is conservative. By standard constructions, see Chapter III in \cite{RW1} and V in \cite{RW2}, this provides a (non-explosive) Markov process with generator $- L_C$, which is itself a solution to \eqref{eq:diff2}.

To wrap up the proof of Theorem \ref{th:onRd0}, note, to begin with, that the existence and uniqueness of a strong solution to \eqref{eq:diff2}, maybe up to some explosion time, is clearly just the local Lipschitz continuity of the coefficients, see Theorem 3.1 in \cite{IW}. By uniqueness, this solution has the same properties as the (weak) one that we just constructed, namely that there is no explosion and that its transition semigroup is $(T_t)$. In particular, $\pi$ is its unique invariant measure, and the generator of $(T_t)$ on $\Lpi$ is $- L_C$.

\subsubsection{Proof of Theorem \ref{th:onRd}} \label{sec:proofthonRd}

Now, to prove Theorem \ref{th:onRd}, we will obviously use Theorem \ref{th:ineqop}, and to this end, we need spectral properties of the generator. Recall \eqref{eq:Lfg} and \eqref{eq:Cdotgrad} and note that, since $\Cinf_c$ is dense in $\Hpi$, and that $L$ and $C \cdot \grad$ are bounded from $\Hpi$ to $\Lpi$ by Lemma \ref{lem:cont}, then these equalities extends to $f,g \in \Hpi$, i.e.
\begin{equation} \label{eq:symRd}
\la L f, g \rapi = \la \grad f, \grad g \rapi = \la f, L g \rapi, \quad \la C \cdot \grad f, g \rapi = - \la f, C \cdot \grad g \rapi, \quad f, g \in \Hpi.
\end{equation}
Hence, $L$ is symmetric and $C \cdot \grad$ is antisymmetric on $\Hpi$. Let us give more properties of $L$ in the following result.

\begin{lemma} \label{lem:LonRd}
Assume \textbf{(A1)}, \textbf{(A2)}, \textbf{(A3)}, \textbf{(A4)}, \textbf{(A5)}, \textbf{(A6)}. Then $L$ is self-adjoint and positive, has a purely discrete spectrum consisting of eigenvalues and a complete set of eigenfunctions.
\end{lemma}

\begin{proof}
It is well-known and easy to check that $L$ is unitarily equivalent to the Schr\"{o}dinger operator $S = -\Delta + V$, where
\[
V = \frac14 |\grad V|^2 - \frac12 \Delta U 
\]
on $\cL^2(\ddx)$, via the transformation $T \, : \, f \mapsto e^{-U/2} f$. To be precise, consider $L$ and $S$ with domain $C^2_c$. Then $T$ is a unitary transformation from $\Lpi$ onto $\cL^2(\ddx)$, $T(C^2_c) = C^2_c$ and $L = T^{-1} S T$ on $C^2_c$. Hence, (the closure of) $L$ and $S$ have the same spectral properties, see e.g. Section 73 in \cite{AkGlaz}.

Now, as mentioned in the proof of Theorem 7.4 in \cite{MetafuneRegularity}, $\Cinf_c$ is a core for $(L,\Hpi)$. Since $\Cinf_c \subset C^2_c \subset \Hpi$, then $C^2_c$ is also a core, so the closure of $(L,C^2_c)$ is our operator $(L,\Hpi)$.

On the other end, the closure $(S,\cD)$ of $(S,C^2_c)$ is clearly $S$ ``defined as a sum of quadratic forms'', in the terminology of \cite{RS2}, see Example 4 p. 181 therein. But \textbf{(A5)} and \textbf{(A6)} imply\footnote{Equivalently, we could have also directly assumed that $V(x) \to \infty$, but we would rather keep the most simple assumptions.} that $V(x) \to \infty$ as $x \to \infty$. Theorem XIII.67 in \cite{RS4} thus ensures that $(S,\cD)$ has a purely discrete spectrum and a complete set of eigenfunctions. By unitary equivalence, so does $L$.
\end{proof}

Without loss of generality, we will assume that our functions have mean zero. This is possible, since by symmetry, we have that $\la L f, \un \rapi = \la f, L \un \rapi = \la f , 0 \rapi = 0$ for all $f \in \Hpi$, and similarly for $C \cdot \grad$ and thus $L_C$. We may thus consider our operators as operators on $\Lpi_0$ with domain $\Hpi_0$. We summarize their properties in the following result.

\begin{lemma} \label{lem:oponRd}
As operators on $\Lpi_0$ with domain $\Hpi_0$,
\begin{itemize}
\item $L$ is self-adjoint, positive, and has a spectral gap;
\item $C \cdot \grad$ is antisymmetric;
\item $L_C$ is closed and has a spectral gap.
\end{itemize}
\end{lemma}

\begin{proof}
For the first point, since $\la L f, f \rapi = \| \grad f \|^2_{\pi}$ for $f \in \Hpi$, then the kernel of $L$ consists of constant functions. As an operator on $\Hpi$, 0 is thus a simple eigenvalue of $L$. From Lemma \ref{lem:LonRd}, as an operator on $\Hpi_0$, $L$ has thus a spectral gap. Now, the second point of the result is just \eqref{eq:symRd}. As for the third point, it is the main result, Theorem 1, of \cite{HwangAD}, that we mentioned in the introduction: when $L$ has a spectral gap, then so does $L_C$ (and the gap is even larger). The closedness is just because it is the generator of a strongly continuous semigroup on its natural domain.
\end{proof}

To conclude, note that the spectral gap property implies that $L$ and $L_{\pm C}$ are bijections from $\Hpi_0$ onto a dense subset of $\Lpi_0$ with bounded inverse. But since they are closed, they even are onto $\Lpi_0$, see p. 209 in \cite{YosidaFA}. Hence, \textbf{(G1)}, \textbf{(G2)} and \textbf{(G3)} hold with $\cH$ the space of real functions of $\Lpi_0$, $\cD$ the space of real functions of $\Hpi_0$, $S = L$ and $A = - C \cdot \grad$, so that Theorem \ref{th:ineqop} can be applied. For the CLT, we just need to check that $\pi$ is ergodic, which is equivalent to the fact just proved that 0 is a simple eigenvalue of $L_C$ on $\Lpi$, see Proposition 2.2 in \cite{BhattaCLT}.

\subsection{On a compact manifold}

\subsubsection{Assumptions and results}

In this section, we consider the diffusions \eqref{eq:diff} and \eqref{eq:diff2} on $M$, a smooth d-dimensional connected compact Riemannian manifold. For the precise meaning of these SDEs, we refer to Chapter 5 in \cite{IW}. We will for simplicity assume that $U$ and $C$ are smooth. Let us define the space
\[
\Wpi = \left \{ f \in \cH^1(\pi), \, Lf \in \Lpi \right \}.
\]
Since $C$ is bounded, we do not need a result like Lemma \ref{lem:cont}, and our operators $L$, $C \cdot \grad$ and $L_C$ make sense as unbounded operators on $\Lpi$, with domain $\Wpi$. We will once again give two results, the first concerning the behavior of the diffusions, the second one being the main result about reducing asymptotic variance.

\begin{theorem} \label{th:onmanifold0}
Assume \textbf{(A1')}, \textbf{(A2)}, \textbf{(A3)}. Then the following hold.
\begin{enumerate}
\item Equation \eqref{eq:diff2} has a unique strong solution.
\item The measure $\pi$ is its unique invariant distribution.
\item The generator of \eqref{eq:diff} on $\Lpi$ is given by $- L$, with domain $\Wpi$.
\item The generator of \eqref{eq:diff2} on $\Lpi$ has a domain containing $\Wpi$, and is equal to $- L_C$ on $\Wpi$. 
\end{enumerate}
\end{theorem}

Note that, unlike Theorem \ref{th:onRd0}, we cannot precisely identify the domain of the perturbed diffusion \eqref{eq:diff2}, but only say that it contains $\Wpi$, which is more than sufficient for all intents and purposes. Our main result is the following.

\begin{theorem} \label{th:onmanifold}
Assume \textbf{(A1')}, \textbf{(A2)}, \textbf{(A3)}. Then $L_C$ is onto $\Lpi_0$, and for $f \in \Lpi_0$ and $h \in \Hpi$ such that $L_C h = f$, the CLT holds:
\[
t^{-1/2} \int_0^t f(X_s) \ds
\]
converges to a centered normal variable with variance $\sigma^2_C(f)$, where
\[
\sigma^2_C(f) = 2 \la f, h \rapi.
\]
Additionally, adding an antisymmetric drift reduces the asymptotic variance, in that
\[
\sigma^2_C(f) \leq \sigma_0^2(f)
\]
for all $f \in \cL^2(\pi)_0$.
\end{theorem}

\subsubsection{Proof of Theorem \ref{th:onmanifold0}}

As in Section \ref{sec:proofthonRd}, \eqref{eq:Lfg} and \eqref{eq:Cdotgrad} can be extended by density, and we have
\begin{equation} \label{eq:symmanif}
\la L f, g \rapi = \la \grad f, \grad g \rapi = \la f, L g \rapi, \quad f, g \in \Wpi
\end{equation}
and
\begin{equation} \label{eq:asymmanif}
\la C \cdot \grad f, g \rapi = - \la f, C \cdot \grad g \rapi, \quad f, g \in \cH^1(\pi).
\end{equation}
Let us summarize more properties of $L$ in the following result.

\begin{lemma} \label{lem:Lonmanif}
The operator $L$, with domain $\Wpi$, is self-adjoint and positive. It has a purely discrete spectrum consisting of eigenvalues. It is the generator of the heat semigroup $(T_t)$ on $\Lpi$.
\end{lemma}

\begin{proof}
The first two parts are Theorem 4.6 and 10.13 in \cite{GrigoryanBook}, and the last one is a direct consequence of Exercises 4.41 and 4.43 therein.
\end{proof}

This allows us to get the generator of our diffusions and complete the proof of Theorem \ref{th:onmanifold0}. To begin with, the first two parts of the result are just Theorem 1.1 and Prop 4.5 in Chapter 5 of \cite{IW}. Let $(G,D(G))$ and $(G_C,D(G_C))$ be the generators of \eqref{eq:diff} and \eqref{eq:diff2} on $\Lpi$. Clearly (see still \cite{IW}), $\Cinf \subset D(G)$, $\Cinf \subset D(G_C)$, and $G = L$, $G_C = L_C$ on $\Cinf$. By density of $\Cinf$ in $\Lpi$, the semigroup of $G$ is $(T_t)$ and thus has domain $\Wpi$ on $\Lpi$, by Lemma \ref{lem:Lonmanif}. Now, $D(G_C)$ is closed for the norm $\|\cdot\|_{G_C} := \normpi{\cdot} + \normpi{G_C \cdot}$, so contains in particular the closure of $\Cinf$ for this norm. But $\|\cdot\|_{G_C}$ is weaker than $\normpi{\cdot}$, and $\Cinf$ is dense in $\Wpi$ for the $\normpi{\cdot}$ norm, so in particular, $\Wpi \subset D(G_C)$.

\subsubsection{Proof of Theorem \ref{th:onmanifold}} \label{sec:proofthonmanifold}

As in Section \ref{sec:proofthonRd}, \eqref{eq:symmanif} and \eqref{eq:asymmanif} imply that our operators take values in $\Lpi_0$. As for Lemma \ref{lem:oponRd}, $L$, as an unbounded operator on $\Lpi_0$ with domain $\Wpi_0$ is thus positive, self-adjoint, and has a spectral gap. Hence, it is a bijection from $\Wpi_0$ onto $\Lpi_0$. However, we cannot say the same for $L_C$, since we would need to prove that it is closed on $\Wpi_0$ and has a spectral gap. The first point is not clear, while the second one would require to adapt the proof of Theorem 1 in \cite{HwangAD}. We shall instead use a workaround and change the spaces and inner products we consider. So let us consider $\cH^1(\pi)_0$ as a Hilbert space with inner product
\[
[f, g] = \la \grad f, \grad g \rapi.
\]
Note that the first equality of \eqref{eq:symmanif} can be extended by density, so that
\begin{equation} \label{eq:symmanif2}
[f, g] = \la L f, g \rapi, \quad f \in \Wpi_0, \; g \in \cH^1(\pi)_0.
\end{equation}
In particular,
\[
[f, f] \geq \l \la f, f \rapi, \quad f \in \Wpi_0
\]
where $\l > 0$ is the smallest eigenvalue of $L$ on $\Wpi_0$. Still by density, this last inequality holds for $f, g \in \cH^1(\pi)_0$, so that the norm induced by $[ \cdot, \cdot ]$ and the usual $\cH^1(\pi)$ norm are equivalent, and $\cH^1(\pi)_0$ with the $[ \cdot, \cdot ]$ inner product is a Hilbert space.

As we just noticed, $L$ is a bijection from $\Wpi_0$ to $\Lpi_0$, and in particular $T = i L^{-1} C \cdot \grad$ is well-defined as an operator on $\cH^1(\pi)_0$. It has the following properties.

\begin{lemma} 
The operator $T$ is bounded and self-adjoint on $\left ( \cH^1(\pi)_0, [ \cdot, \cdot ] \right )$.
\end{lemma}

\begin{proof}
For $f,g \in \cH^1(\pi)_0$, we readily have from \eqref{eq:asymmanif} and \eqref{eq:symmanif2}, since $T f, T g \in \Wpi_0$,
\[
[ T f, g ] = - i \la C \cdot \grad f, g \rapi = i \la f, C \cdot \grad g \rapi
= \la f, L T g \rapi = [f , T g].
\]
Now, take
\[
m = \max \left ( \sup_{x \in M} |C(x)|, \|L^{-1}\|_{\pi} \right ).
\]
Then, by \eqref{eq:symmanif2} again,
\begin{align*}
[T f, T f] & = \la C \cdot \grad f, L^{-1} C \cdot \grad f \rapi \\
& \leq \| C \cdot \grad f \|_{\pi} \| L^{-1} C \cdot \grad f\|_{\pi} \\
& \leq m^3 \|\grad f\|^2_{\pi} \\
& = m^3 [f, f].
\end{align*}
Therefore, $T$ is bounded. Alternatively, one may just invoke the Hellinger-Toeplitz theorem.
\end{proof}

In particular, by the spectral theorem, we may define the resolvent $(I - L^{-1} C \cdot \grad)^{-1} 
= (I + i T)^{-1}$ on $\cH_0^1$. It enjoys the following property.

\begin{lemma} 
The operator $L - C \cdot \grad$ is a bijection from $\Wpi_0$ to $\Lpi_0$, 
with inverse $(I + i T)^{-1} L^{-1}$.
\end{lemma}

\begin{proof}
To begin with, note that $(I + i T)^{-1} L^{-1}$ is well-defined on $\Lpi_0$. On the other hand, one may write, on $\Wpi_0$,
\[
(L - C \cdot \grad) = L (I - L^{-1} C \cdot \grad) = L(I + i T)
\]
and the result follows.
\end{proof}

This allows to conclude as for the proof of Theorem \ref{th:onRd}. Let us however give an easier version of the computation of Section \ref{sec:maincomp}, using the $[\cdot, \cdot]$ inner product. Consider $L$ and $L_C$ as bijections from $\Wpi_0$ to $\Lpi_0$. Let $f \in \Lpi_0$ be real and $g = L^{-1} f$. By the spectral theorem for bounded operators, namely $T$ in $(\Hpi_0,[\cdot, \cdot])$, we can define $\mu_g$, the spectral measure of $T$ associated with $g$. Then
\begin{align*}
\sigma^2_C(f) & = 2 \la L_C^{-1} f, f \rapi \\
& = 2 \la (I + i T)^{-1} L^{-1} f, f \rapi \\
& = 2 \la (I + i T)^{-1} g, L g \rapi \\
& = 2 \left [ (I + i T)^{-1} g, g \right ] \\
& = 2 \int_{\sigma(T)} \frac{1}{1 + i y} \: \mu_g(dy) \\
& = 2 \int_{\sigma(T)} \frac{1}{1 + y^2} \: \mu_g(dy) \\
& \leq 2 \int_{\sigma(T)} 1 \: \mu_g(dy) \\
& = 2 [g,g] = 2 \la f, L^{-1} f \rapi = \sigma_0^2(f).
\end{align*}
Note that this avoids using Theorem \ref{th:ineqop}, and provides a slightly different formula for the asymptotic variances.

\section{Extensions} \label{sec:extensions}

Let us now still consider the diffusions \eqref{eq:diff} and \eqref{eq:diff2} on $M = \R^d$ 
or a compact Riemannian manifold under the assumptions of Theorem \ref{th:onRd} or \ref{th:onmanifold}. 
We have shown in the previous sections that these assumptions ensure that, for any $f \in \Lpi_0$,
\begin{equation} \label{eq:var}
\sigma_C^2(f) = 2 \int_{\sigma(B)} \frac{1}{1+y^2} \: \mu_g(dy) \leq
2 \int_{\sigma(B)} 1 \: \mu_g(dy) = 2 \|g\|^2_{\pi} 
= 2 \|L^{-1/2} f\|^2_{\pi} = \sigma_0^2(f),
\end{equation}
where $B = i L^{-1/2} (C \cdot \grad) L^{-1/2}$, $g = L^{-1/2} f$, and $\mu_g$ is the spectral measure of $B$ associated with the vector $g$. In this section, we will use Formula \eqref{eq:var} to derive more detailed results concerning the variance reduction. Because of the last remark in Section \ref{sec:proofthonmanifold}, we could also formulate them slightly differently in the compact manifold case. However, we shall only use Formula \eqref{eq:var} since it works on both $\R^d$ and a compact manifold.

\subsection{Equality}

Let us first give a condition for equality in \eqref{eq:var}.

\begin{cor} \label{cor:ineqop}
Under the assumptions of Theorem \ref{th:onRd} (resp. Theorem \ref{th:onmanifold}), we have $\sigma_C^2(f) = \sigma_0^2(f)$ if and only if $f \in L(\Ker(C \cdot \grad))$. In particular, $\sigma_C^2(f) < \sigma_0^2(f)$ for all nonzero $f \in \Lpi_0$ when $C \cdot \grad$ is injective on $\Hpi_0$ (resp. $\Wpi_0$).
\end{cor}

\begin{proof}
From \eqref{eq:var}, $\sigma_C^2(f) = \sigma_0^2(f)$ if and only if $\mu_g$ puts all its mass at zero, which means, by definition (see \cite{RS1}), that $g \in \Ker(B)$, i.e. $C \cdot \grad L^{-1} f = 0$, whence the first part follows. The second part only uses that $L^{-1}$ has range $\Hpi_0$ (resp. $\Wpi_0$).
\end{proof}

Therefore, one would prefer a $C$ such that $\Ker(C \cdot \grad) = \{0\}$, even more considering the following results. However, we do not know how to find such a $C$, or even if one necessarily exists.

\subsection{Worst-case analysis}

A direct consequence of Theorems \ref{th:onRd} and \ref{th:onmanifold} is the worst-case analysis comparison
\begin{equation} \label{eq:worstcase0}
\sup_{\|f\|_{\pi} = 1} \sigma_C^2(f) \leq \sup_{\|f\|_{\pi} =1} \sigma_0^2(f),
\end{equation}
where the sup is over the real $f \in \Lpi_0$. Now, recall from \eqref{eq:var} that $\sigma_0^2(f) = 2 \la L^{-1} f, f \rapi$, so that the RHS of \eqref{eq:worstcase0} is $2 \|L^{-1}\|_{\pi}$. Since $L$ is self-adjoint on $\Hpi_0$ from Lemmas \ref{lem:LonRd} and \ref{lem:Lonmanif}, then $ \|L^{-1}\|_{\pi} = 1 / \l$, where $\l$ is the smallest eigenvalue of $L$. The following result provides a condition such that the irreversible algorithm performs strictly better than the reversible algorithm in the worst possible situation. It is worth mentioning that this result is similar to Theorem 1 in \cite{HwangAD}.

\begin{theorem}
Under the assumptions of Theorem \ref{th:onRd} or \ref{th:onmanifold}, if
\[
\Ker(L - \l) \cap \Ker(C \cdot \grad) = \{0\},
\]
then
\begin{equation} \label{eq:worstcase}
\sup_{\|f\|_{\pi} = 1} \sigma_C^2(f) < \sup_{\|f\|_{\pi} =1} \sigma_0^2(f) = \frac{2}{\l} .
\end{equation}
\end{theorem}

\begin{proof}
Assume that there is equality in \eqref{eq:worstcase}, and recall that $\sigma_C^2(f) = 2 \la L_C^{-1} f, f \rapi$. Then there is a normalized sequence $(f_n)$ of real functions in $\Lpi_0$ such that
\begin{equation} \label{eq:deffn}
\la L_C^{-1} f_n, f_n \rapi \to \frac{1}{\l}.
\end{equation}
From Lemmas \ref{lem:LonRd} and \ref{lem:Lonmanif}, $L$ has a purely discrete spectrum with a complete family of eigenvalues, so that $\l > 0$ is an eigenvalue of $L$, $\Ker(L - \l)$ is finite-dimensional, and for some $\delta > 0$, the spectrum of $L$ restricted to $\Ker(L - \l)^{\perp}$ is bounded below by $\l + \delta$. In particular
\[
\left \| L^{-1}_{|\Ker(L-\l)^{\perp}} \right \|_{\pi} \leq \frac{1}{\l + \dl}.
\]

Let us write $f_n = f_n^1 + f_n^2$, with $f_n^1 \in \Ker(L - \l)$ and $f_n^2 \in \Ker(L - \l)^{\perp}$. Since $L^{-1}$ is self-adjoint, and by the preceding remark, we have
\begin{align*}
\la L^{-1} f_n, f_n \rapi & = \la L^{-1} f_n^1, f_n^1 \rapi + \la L^{-1} f_n^2, f_n^2 \rapi + 2 \la L^{-1} f_n^1, f_n^2 \rapi \\
& = \frac{1}{\l} \|f_n^1\|^2_{\pi} + \la L^{-1} f_n^2, f_n^2 \rapi + \frac{2}{\l} \la f_n^1, f_n^2 \rapi \\
& = \frac{1}{\l} \|f_n^1\|^2_{\pi} + \la L^{-1} f_n^2, f_n^2 \rapi \\
& \leq \frac{1}{\l} \|f_n^1\|^2_{\pi} + \frac{1}{\l + \delta} \|f_n^2\|^2_{\pi} \\
& = \frac{1}{\l} \|f_n^1\|^2_{\pi} + \frac{1}{\l + \delta} (1 - \|f_n^1\|^2_{\pi}).
\end{align*}
But, on the other hand,
\[
\la L^{-1} f_n, f_n \rapi \geq \la L_C^{-1} f_n, f_n \rapi
\]
by our main results, so that
\begin{equation} \label{eq:ineqfn}
\frac{1}{\l} \|f_n^1\|^2_{\pi} + \frac{1}{\l + \delta} (1 - \|f_n^1\|^2_{\pi}) \geq \la L_C^{-1} f_n, f_n \rapi.
\end{equation}
Now, since $\Ker(L - \l)$ is finite-dimensional and $(f_n^1)$ is bounded, $(f_n^1)$ converges up to some subsequence, which we do not write. Then, \eqref{eq:deffn} and \eqref{eq:ineqfn} imply that necessarily $\|f_n^1\|_{\pi} \to 1$, so that $f_n^2 \to 0$, and $f_n$ itself converges to some nonzero
\[
f = \lim f_n^1 \in \Ker(L - \l).
\]
Passing to the limit in \eqref{eq:deffn} finally implies that
\[
\la L_C^{-1} f, f \rapi = \frac{1}{\l}.
\]
Hence, $\sigma_C^2(f) = \sigma_0^2(f)$. By Corollary \ref{cor:ineqop}, this means $L^{-1} f = f / \l \in \Ker(C \cdot \grad)$. Hence, $f \in \Ker(C \cdot \grad) \cap \Ker(L - \l)$, which is not allowed by assumption.
\end{proof}

\subsection{Growing Perturbation}

Let us conclude with a result concerning the limiting behavior of the asymptotic variance as the magnitude of the perturbation goes to infinity. The idea of growing perturbation in diffusion acceleration goes back at least to \cite{HwangAGD}. In \cite{FrankeBSG}, the behavior of the spectral gap of the Laplace-Beltrami operator on a compact Riemannian manifold, perturbed by a growing drift, is investigated; see also \cite{HwangABM}. Define here $P$ to be the projection on $\Ker B$. By ``growing drift'', we mean that we will consider the behavior of $\sigma^2_{kC}(f)$ as $k \to \infty$.

\begin{theorem} \label{th:growing}
For any $f \in \Lpi$, $\s^2_{k C}(f)$ is decreasing in $k$ and
\[
\lim_{k \to \infty} \sigma^2_{kC}(f) = 2 \| P L^{-1/2} f \|^2.
\]
In particular, if $\Ker(C \cdot \grad) = \{0\}$, then we can push the asymptotic variance to 0, in that
\[
\lim_{k \to \infty} \sigma^2_{kC}(f) = 0.
\]
If $B$ has a spectral gap, then the asymptotic variance goes uniformly to 0, i.e.,
\[
\lim_{k \to \infty} \sup_{\|f\|_{\pi} = 1} \sigma^2_{kC}(f) = 0.
\]
\end{theorem}

It is worth contrasting this result with those of \cite{FrankeBSG,HwangAGD,HwangABM}. In \cite{HwangAGD}, it is shown that the spectral gap does not need to be increasing in $k$, and that the smallest spectral gap can be attained for some finite $k$. In \cite{HwangABM}, antisymmetric perturbations of the Laplace-Beltrami operator on the torus are considered, and the authors prove that the spectral gap as a finite limit as $k \to \pinf$, but can be made arbitrarily close to 0 for well-chosen $C$. In \cite{FrankeBSG}, the limit of the spectral gap as $k \to \pinf$ is investigated, and shown to be infinite, except if a very strict condition is verified. We ignore whether the condition that $B$ has a spectral gap in our result is restrictive.

\begin{proof}
Note first that the results of Theorem \ref{th:onRd} and \ref{th:onmanifold} obviously apply to $k C$. By \eqref{eq:var}, we have that
\[
\sigma^2_{k C}(f) = 2 \int_{\sigma(B)} \frac{1}{1 + k^2 y^2} \mu_{g}(dy)
\]
with $g = L^{-1/2} f$. Hence, by bounded convergence,
\[
\lim_{k \to \infty} \sigma^2_{k C}(f) = 2 \mu_g(\{0\}),
\]
and $\mu_g(\{0\})$, by definition, is just the squared norm of the projection of $g$ on $\Ker B$, i.e. of $P g = P L^{-1/2} f$. Since $L^{-1/2}$ is injective, $B$ is injective if and only if $C \cdot \grad$ is, and hence the second part of the result follows. Finally, if $B$ has a spectral gap, then $\mu_g$ has no mass in a neighborhood of 0, say $[-\delta,\delta]$, so that
\begin{align*}
\sigma^2_{k C}(f) & = 2 \int_{\sigma(A)} \frac{1}{1 + k^2 y^2} \: \mu_{g}(dy) \\
& \leq \frac{2}{1 + k^2 \delta^2} \int 1 \: \mu_g(\ddy) \\
& = \frac{2}{1 + k^2 \delta^2} \|g\|_{\pi}^2 \\
& \leq \frac{2}{1 + k^2 \delta^2} \|L^{-1/2}\|^2_{\pi} \|f\|^2_{\pi},
\end{align*}
and the last part of the statement follows.
\end{proof}

\section{Conclusion} \label{sec:concl}

Our results mainly provide qualitative information on the asymptotic variance, with less quantitative information due to the difficulty encountered in the manipulation of spectral measures. Of course, within the family of all possible choices of $C$ (with unit norm, say), one would like the one that gives the smaller asymptotic variance, and maybe even the value of this lower bound. The theoretical existence and practical construction of such a perturbation $C$ would of course be of great interest. In a similar direction, the results of Section \ref{sec:extensions} show that a $C$ with $\Ker(C \cdot \grad) = \{0\}$ is preferable. As mentioned, we do not know if such a $C$ always exists, and, even more to the point, if it can be constructed. The same goes for finding a $B$ with a spectral gap as in Theorem \ref{th:growing}.

In the case of an Ornstein-Uhlenbeck process, where $U$ is quadratic, it is reasonable to only consider $C$ of the form $Q \grad U$, for an antisymmetric matrix $Q$. The existence of a best $C$ of this type is then obvious, and it would be interesting to get a closed form for it. However, one cannot have $\Ker(C \cdot \grad) = \{0\}$, so that $\s^2_{k C}(f)$ has a positive limit as $k \to \pinf$. A closed form for this limit might shed some light on the open questions considered above.

\paragraph{Acknowledgments} The authors would like to thank Brice Franke for fruitful discussions and Yi-Ching Yao for his comments.

\bibliographystyle{abbrv}
\bibliography{Bibli}

\begin{thebibliography}{10}

\bibitem{AkGlaz}
N.~I. Akhiezer and I.~M. Glazman.
\newblock {\em Theory of linear operators in {H}ilbert space}.
\newblock Dover Publications, Inc., New York, 1993.
\newblock Translated from the Russian and with a preface by Merlynd Nestell,
  Reprint of the 1961 and 1963 translations, Two volumes bound as one.

\bibitem{BhattaCLT}
R.~N. Bhattacharya.
\newblock On the functional central limit theorem and the law of the iterated
  logarithm for {M}arkov processes.
\newblock {\em Z. Wahrsch. Verw. Gebiete}, 60(2):185--201, 1982.

\bibitem{BrooksMCMC}
S.~Brooks.
\newblock Markov chain monte carlo method and its application.
\newblock {\em Journal of the Royal Statistical Society: Series D (The
  Statistician)}, 47(1):69--100, 1998.

\bibitem{ChenOTM}
T.-L. Chen, W.-K. Chen, C.-R. Hwang, and H.-M. Pai.
\newblock On the optimal transition matrix for {M}arkov chain {M}onte {C}arlo
  sampling.
\newblock {\em SIAM J. Control Optim.}, 50(5):2743--2762, 2012.

\bibitem{ChenARMC}
T.-L. Chen and C.-R. Hwang.
\newblock Accelerating reversible markov chains.
\newblock {\em Stat. and Probab. Letters}, 83:1956--1962, 2013.

\bibitem{FrankeBSG}
B.~Franke, C.-R. Hwang, H.-M. Pai, and S.-J. Sheu.
\newblock The behavior of the spectral gap under growing drift.
\newblock {\em Trans. Amer. Math. Soc.}, 362(3):1325--1350, 2010.

\bibitem{FrigessiOSS}
A.~Frigessi, C.-R. Hwang, and L.~Younes.
\newblock Optimal spectral structure of reversible stochastic matrices, {M}onte
  {C}arlo methods and the simulation of {M}arkov random fields.
\newblock {\em Ann. Appl. Probab.}, 2(3):610--628, 1992.

\bibitem{GrigoryanBook}
A.~Grigor'yan.
\newblock {\em Heat kernel and analysis on manifolds}, volume~47 of {\em AMS/IP
  Studies in Advanced Mathematics}.
\newblock American Mathematical Society, Providence, RI, 2009.

\bibitem{HastingsMCMC}
W.~K. Hastings.
\newblock Monte carlo sampling methods using markov chains and their
  applications.
\newblock {\em Biometrika}, 57(1):97--109, 1970.

\bibitem{HwangAGD}
C.-R. Hwang, S.-Y. Hwang-Ma, and S.~J. Sheu.
\newblock Accelerating {G}aussian diffusions.
\newblock {\em Ann. Appl. Probab.}, 3(3):897--913, 1993.

\bibitem{HwangAD}
C.-R. Hwang, S.-Y. Hwang-Ma, and S.-J. Sheu.
\newblock Accelerating diffusions.
\newblock {\em Ann. Appl. Probab.}, 15(2):1433--1444, 2005.

\bibitem{HwangABM}
C.-R. Hwang and H.-M. Pai.
\newblock Accelerating brownian motion on n-torus.
\newblock {\em Statistics \& Probability Letters}, 83(5):1443--1447, 2013.

\bibitem{IW}
N.~Ikeda and S.~Watanabe.
\newblock {\em Stochastic differential equations and diffusion processes},
  volume~24 of {\em North-Holland Mathematical Library}.
\newblock North-Holland Publishing Co., Amsterdam; Kodansha, Ltd., Tokyo,
  second edition, 1989.

\bibitem{Komorowski}
T.~Komorowski, C.~Landim, and S.~Olla.
\newblock {\em Fluctuations in {M}arkov processes}, volume 345 of {\em
  Grundlehren der Mathematischen Wissenschaften [Fundamental Principles of
  Mathematical Sciences]}.
\newblock Springer, Heidelberg, 2012.
\newblock Time symmetry and martingale approximation.

\bibitem{Lunardi}
A.~Lunardi.
\newblock On the {O}rnstein-{U}hlenbeck operator in {$L^2$} spaces with respect
  to invariant measures.
\newblock {\em Trans. Amer. Math. Soc.}, 349(1):155--169, 1997.

\bibitem{MetafuneRegularity}
G.~Metafune, J.~Pr{\"u}ss, R.~Schnaubelt, and A.~Rhandi.
\newblock {$L^p$}-regularity for elliptic operators with unbounded
  coefficients.
\newblock {\em Adv. Differential Equations}, 10(10):1131--1164, 2005.

\bibitem{MiraEFSS}
A.~Mira.
\newblock Efficiency of finite state space {M}onte {C}arlo {M}arkov chains.
\newblock {\em Statist. Probab. Lett.}, 54(4):405--411, 2001.

\bibitem{MiraOIP}
A.~Mira.
\newblock Ordering and improving the performance of {M}onte {C}arlo {M}arkov
  chains.
\newblock {\em Statist. Sci.}, 16(4):340--350, 2001.

\bibitem{MiraNRMC}
A.~Mira and C.~J. Geyer.
\newblock On non-reversible {M}arkov chains.
\newblock In {\em Monte {C}arlo methods ({T}oronto, {ON}, 1998)}, volume~26 of
  {\em Fields Inst. Commun.}, pages 95--110. Amer. Math. Soc., Providence, RI,
  2000.

\bibitem{PeskunOMCS}
P.~H. Peskun.
\newblock Optimum {M}onte-{C}arlo sampling using {M}arkov chains.
\newblock {\em Biometrika}, 60:607--612, 1973.

\bibitem{RS2}
M.~Reed and B.~Simon.
\newblock {\em Methods of modern mathematical physics. {II}. {F}ourier
  analysis, self-adjointness}.
\newblock Academic Press [Harcourt Brace Jovanovich, Publishers], New
  York-London, 1975.

\bibitem{RS4}
M.~Reed and B.~Simon.
\newblock {\em Methods of modern mathematical physics. {IV}. {A}nalysis of
  operators}.
\newblock Academic Press [Harcourt Brace Jovanovich, Publishers], New
  York-London, 1978.

\bibitem{RS1}
M.~Reed and B.~Simon.
\newblock {\em Methods of modern mathematical physics. {I}. {F}unctional
  analysis}.
\newblock Academic Press, Inc. [Harcourt Brace Jovanovich, Publishers], New
  York, second edition, 1980.

\bibitem{Rey-BelletILS}
L.~Rey-Bellet and K.~Spiliopoulos.
\newblock Irreversible langevin samplers and variance reduction: a large
  deviation approach.
\newblock {\em arXiv preprint arXiv:1404.0105}, 2014.

\bibitem{RW1}
L.~C.~G. Rogers and D.~Williams.
\newblock {\em Diffusions, {M}arkov processes, and martingales. {V}ol. 1}.
\newblock Cambridge Mathematical Library. Cambridge University Press,
  Cambridge, 2000.
\newblock Foundations, Reprint of the second (1994) edition.

\bibitem{RW2}
L.~C.~G. Rogers and D.~Williams.
\newblock {\em Diffusions, {M}arkov processes, and martingales. {V}ol. 2}.
\newblock Cambridge Mathematical Library. Cambridge University Press,
  Cambridge, 2000.
\newblock It{\^o} calculus, Reprint of the second (1994) edition.

\bibitem{YosidaFA}
K.~Yosida.
\newblock {\em Functional analysis}.
\newblock Classics in Mathematics. Springer-Verlag, Berlin, 1995.
\newblock Reprint of the sixth (1980) edition.

\end{thebibliography}

\end{document}